\documentclass[11pt, oneside, reqno]{amsart}
\usepackage{fullpage}
\usepackage{amssymb,latexsym,amsmath,verbatim,layout,hyperref,amsthm,color,graphicx,mathabx,mdframed,tikz}  
\numberwithin{equation}{section}
\usepackage[T1]{fontenc}
\usepackage[light,math]{anttor}
\usepackage {mathrsfs}
\usepackage{hyperref}
\usepackage{cleveref}
\usepackage{scalerel}

\newcommand\reallywidehat[1]{\arraycolsep=0pt\relax%
\begin{array}{c}
\stretchto{
  \scaleto{
        \scalerel*[\widthof{\ensuremath{#1}}]{\kern-.5pt\bigwedge\kern-.5pt}
    {\rule[-\textheight/2]{1ex}{\textheight}} %WIDTH-LIMITED BIG WEDGE
  }{\textheight} % 
}{0.5ex}\\           % THIS SQUEEZES THE WEDGE TO 0.5ex HEIGHT
#1\\                 % THIS STACKS THE WEDGE ATOP THE ARGUMENT
\rule{-1ex}{0ex}
\end{array}
}             
\theoremstyle{plain}% default
\newtheorem{thm}{Theorem}[section]
\newtheorem{lem}[thm]{Lemma}

\newtheorem*{cor}{Corollary}

\theoremstyle{definition}

\theoremstyle{remark}
\newtheorem*{rmk}{Remark}

\newtheorem{eg}{Example}
\usepackage{hyperref}

\newcommand{\R}{\mathbb{R}}

\newcommand{\A}{\mathcal{A}}

\newcommand{\rectFunc}{\text{rect}}

\title{An explicit numerical algorithm to the solution of Volterra integral equation of the second kind
}

\author{Leanne Dong\\ 
Behavioural Data Science Group\\
Faculty of Engineering and IT, The University of Technology Sydney \\
Ultimo NSW 2007, Australia\\
leanne.dong@uts.edu.au\\
\and\\
John van der Hoek\\
School of Mathematics and Statistics, The University of South Australia\\
John.vanderHoek@unisa.edu.au
}

\date{\today}
\begin{document}
	\maketitle

\begin{abstract}
This paper considers a numeric algorithm to solve the equation
\begin{align*}
    y(t)=f(t)+\int^t_0 g(t-\tau)y(\tau)\,d\tau
\end{align*}
with a kernel $g$ and input $f$ for $y$. In some applications we have a smooth integrable kernel but the input $f$ could be a generalised function, which could involve the Dirac distribution. We call the case when $f=\delta$, the Dirac distribution centred at 0, the fundamental solution $E$, and show that $E=\delta+h$ where $h$ is integrable and solve
\begin{align*}
    h(t)=g(t)+\int^t_0 g(t-\tau)h(\tau)\,d\tau
\end{align*}
The solution of the general case is then
\begin{align*}
    y(t)=f(t)+(h*f)(t)
\end{align*}
which involves the convolution of $h$ and $f$. We can approximate $g$ to desired accuracy with piecewise constant kernel for which the solution $h$ is known explicitly. We supply an algorithm for the solution of the integral equation with specified accuracy.
\end{abstract}

\section{Volterra Integral Equation of the Second Kind}

Applications of Hawkes process in various grounds, such as in quantitative finance and machine learning (See for instance \cite{Daley:2003,Rizoiu2017,dassios:2011}) requires one to study Volterra Equation of the second kind when $f=\delta$, where $\delta$ is the Dirac distribution centred at $0$. We call the solution $y(t)$ the fundamental solution of the second order Volterra equation. Using the theory of distributions by Schwartz. One can show that the fundamental solution have the form $\delta+h$ where $h$ is a $L^1$ function solving the equation $h=g+h*g$. It also follow from this study that the general solution of volterra equation is $y=f+h*f$ and this convolution is well defined for many examples used in the studies of Hawkes.

\begin{align*}
    y(t)&=f(t)+\left\{h*(\sum^N_{i=1}w_i\delta(\cdot-t_i)+f_1)\right\}(t)\notag\\
    &=f(t)+\sum^N_{i=1}w_i h(t_i)+(h*f_1)(t)
\end{align*}
We seek solution of 

\begin{align}\label{eq1}
    y(t)=f(t)+\int^t_0 y(\tau)g(t-\tau)\,d\tau,\quad\text{for}\quad t\ge 0
\end{align}
In the Hawkes's setup,
\begin{align}\label{eq2}
    g(t)=k\varphi(t)
\end{align}
where $\varphi$ is $L^1_{+}(\R)$ with norm $\|\cdot\|_1$, that is $\varphi(t)\ge 0$ if $t\ge 0$ and $\varphi(t)=0$ if $t<0$ and $\int^{\infty}_0 \varphi(t)\,dt=1$, $0<k<1$. We seek $h$

\begin{align}\label{eq3}
    h(t)=g(t)+\int^t_0 h(\tau) g(t-\tau)\,d\tau,\quad t\ge 0
\end{align}
In fact
\begin{align}
    h(t)&=g(t)+(g*g)(t)+(g*g*g)(t)+\cdots\notag\\
    &=\sum^{\infty}_{n=1}g^{\otimes n}(t)
\end{align}
where $g^{\otimes n}(t)=g\otimes g\otimes g\cdots\otimes g$ is $n$-fold convolution, $g^{\otimes 1}=g.$

\begin{rmk}
\begin{itemize}
    \item The $n$-fold convolution always exists in $L^1_{+}(\R)$ and
    \begin{align}
        \|g^{\otimes n}\|_{L^1_{+}(\R)}=k^n\|\varphi\|^n_{L^1_{+}(\R)}
    \end{align}
    \item In very few cases analytic expressions for $h$ are available. If $\varphi(t)=\theta e^{-\theta t}$ for $t\ge 0$, then $h(t)=k\theta e^{-(1-k)\theta t}$ for $t\ge 0$. Analytic expressions are also available for 
    
    \begin{align*}
        \varphi(t)=\frac{\beta^{\alpha}}{\Gamma(\alpha)}t^{\alpha-1}e^{-\beta t}\quad\text{for}\quad t > 0, \,\,\alpha>0,\,\,\beta>0
    \end{align*}
    
    \begin{align*}
    \varphi(t)&=
        \begin{cases}
        1 & 0\le t<1\\
        0 & t\ge 1
        \end{cases}
    \end{align*}
    For other choices, numerical procedure are needed.
    
    \item If $g_a$ is an approximation of $g$ and $\|g\|_1$, $\|g_a\|_1\le k<1$, then
    \begin{align}\label{eq6}
        \|h-h_a\|_1\le \frac{1}{(1-k)^2}\|g-g_a\|_1
    \end{align}
     \begin{align}\label{eq7}
        \|h-h_a\|_{\infty}\le \frac{1}{(1-k)^2}\|g-g_a\|_{\infty}
    \end{align}   
where $\|h\|_1=\int^{\infty}_0|h(t)|\,dt$ and $\|h\|_{\infty}=\sup_{t\ge 0}|h(t)|$. The proofs of (\ref{eq6}) and (\ref{eq7}) are in appendix.
\end{itemize}
\end{rmk}
The key result here would be to find $h$ explicitly where 
\begin{align}\label{kernelgb}
    g(t)&=
\begin{cases}
\beta_j& j\delta\le t<(j+1)\delta\\
0 & t<0
\end{cases}
\end{align}
and

\begin{align}\label{intkernelgb}
    \int^{\infty}_0|g(t)|\,dt=\delta\sum^{\infty}_{j=0}|\beta_j|<1
\end{align}
In our explicit calculation here, the $\beta_j\ge 0$. We provide an algorithm for $h$, which involves no approximation.
    
\section{Part I : \textit{Special Case}} 
We study the case
\begin{align}\label{intkernelga}
	g(t)&=
\begin{cases}
	\alpha_j, &\text{ if }j\le t<j+1\\
	0, &\text{  }t<0
\end{cases}
\end{align}
and then derive the case in (\ref{intkernelgb}) from it. 

Define

\begin{align}\label{rect}
\text{rect}(t)&=
\begin{cases}
	1, & 0\le t<1\\
	0, & t\ge 1, t<0
\end{cases}	
\end{align}

\begin{align}\label{Lf}
	(Lf)(t)=f(t-1)
\end{align}
when $f\in L^1_{+}(\R)$. In (\ref{Lf}) we note that $(Lf)(t)=0$ if $t<1$. We use the Laplace transform
\begin{align}\label{ftrect}
	\widehat{\text{rect}}(s)&=\int^{\infty}_0 e^{-st}\text{rect}(t)\,dt\notag\\
	&=\frac{1-e^{-s}}{s}\quad (s>0)   
\end{align}

\begin{align}\label{eq14}
	\widehat{Lf}(s)&=\int^{\infty}_0 Lf(t)e^{-st}\,dt\notag\\
	&=\int^{\infty}_0 f(t-1)e^{-st}\,dt\notag\\
	&=\int^{\infty}_1 f(t-1)e^{-st}\,dt\notag\\
	&=\int^{\infty}_0 f(u)e^{-s(u+1)}\,du\notag\\
	&= e^{-s}\hat{f}(s)\notag\\
	&=\widehat{M}_L(s)\hat{f}(s)
\end{align}
where $\widehat{M}_L(s)$ is called a multiplier for obvious reasons. Let us define
\begin{align}
	\A=\sum^{\infty}_{j=0}a_j L^j
\end{align}
where $L^j=L\circ L\circ\cdots\circ L$ ($j$ times) and $L^j f(t)=f(t-j)$ for $t\ge 0$. So
\begin{align}
	\A f(t)=\sum^{\infty}_{j=0}a_j f(t-j)
\end{align}

\begin{align}
	\widehat{\A f}(s)&=\left(\sum^{\infty}_{j=0}a_j e^{-js}\right)\hat{f}(s)\notag\\
	&=\widehat{M}_{\A}(s)\hat{f}(s)
\end{align}
where $\widehat{M}_{\A}$ is the multiple of $\A$.

Define

\begin{align}\label{Delta}
	\Delta=I-L
\end{align}
So
\begin{align*}
	(\Delta f)(t)=f(t)-f(t-1)
\end{align*}
and $\widehat{\Delta}f(s)=(1-e^{-s})\hat{f}(s)=\widehat{M}_{\Delta}(s)\hat{f}(s)$.

Define
\begin{align}
	Jf(t)=\int^t_0 f(\tau)\,d\tau
\end{align}
Hence
\begin{align*}
	\widehat{Jf}(s)=\frac{1}{s}\hat{f}(s)=\widehat{M}_J(s)\hat{f}(s)
\end{align*}

\begin{lem}
	\begin{align}
		g= \A \rectFunc
	\end{align}
\end{lem}

\begin{proof}
\begin{align*}
&\	\A \rectFunc(t)\\
= &\ \sum^{\infty}_{j=0}a_j (L^j\rectFunc)(t)\\
= &\ \sum^{\infty}_{j=1}a_j\rectFunc(t-j)
\end{align*}
and \begin{align*}
	\rectFunc(t-j)&=
	\begin{cases}
		1, &\text{ if }j\le t<j+1\\
		0, &\text{ otherwise }
	\end{cases}
\end{align*}
\end{proof}
We now compute $g\otimes\cdots\otimes g=g^{\otimes n}$: we note that

\begin{align}
	\widehat{g^{\otimes n}}(s)&=\hat{g}(s)^n\notag\\
	&=\widehat{M}_{\A}(s)^n\widehat{\rectFunc}(s)^n
\end{align}
becomes
\begin{align}
	\hat{g}(s)&=\widehat{M}_{\A}(s)\widehat{\rectFunc}(s)
\end{align}
and
\begin{align*}
	\widehat{\rectFunc}(s)&=\frac{1-e^{-s}}{s}=\widehat{M}_J(s)\widehat{M}_{\Delta}(s)
\end{align*}

\begin{lem}\label{lem2}
	Let 
\begin{align*}
	f(t)&=(t-a)^{\gamma}_{+}\\
	&=
\begin{cases}
	(t-a)^{\gamma}, &\text{ if }t\ge a\\
	0, &\text{ Otherwise}
\end{cases}
\end{align*}
then 
\begin{align}
	\hat{f}(s)&=\frac{\Gamma(\gamma+1)}{s^{\gamma+1}}e^{-sa}
\end{align}
\end{lem}

\begin{proof}
	We have
	\begin{align*}
		\hat{f}(s)&=\int^{\infty}_0 f(t)e^{-st}\,dt\qquad s>0\\
		&=\int^{\infty}_a (t-a)^{\gamma}e^{-st}\,dt\\
		&=\int^{\infty}_0 u^{\gamma}e^{-s(u+a)}\,du\quad u=t-a\\
		&=e^{-sa}\int^{\infty}_0 u^{\gamma}e^{-su}\,du\\
		&=\frac{e^{-sa}\Gamma(\gamma+1)}{s^{\gamma+1}}\quad s>0
	\end{align*}
\end{proof}
where $\Gamma$ is the usual Gamma function, namely $\Gamma(x)=\int^{\infty}_0 t^{x-1}e^{-t}dt$ for $x>0$. Using binomial expansion $(a+b)^n=a^n+\binom{n}{1}a^{n-1}b+\binom{n}{2}a^{n-2}b^2+\cdots+b^n$, we found that 

\begin{align*}
	\widehat{\rectFunc}(s)^n&=\frac{1}{s^n}(1-e^{-s})^n\\
	&=\frac{1}{s^n}\sum^n_{r=0}\binom{n}{r}(-1)^r e^{-rs}\\
	&=\sum^n_{r=0}(-1)^r\binom{n}{r}\left(\frac{1}{s^n}e^{-rs}\right)\\
	&=\sum^n_{r=0}(-1)^r\binom{n}{r}\frac{1}{(n-1)!}\left(\frac{(n-1)!}{s^n}e^{-rs}\right)
\end{align*}
This leads to:
\begin{lem}
\begin{align}\label{nconvrect}
&\	\rectFunc*\rectFunc*\cdots*\rectFunc(t)\notag\\
= &\ \sum^{n}_{r=0}(-1)^r\binom{n}{r}\frac{1}{(n-1)!}(t-r)^{n-1}_+
\end{align}
for $n\ge 1$
\end{lem}
\begin{proof}
	This is an immediate result from our previous calculation.
\end{proof}
Here are some graphics for $n=1$, $2$ and $3$.
%\newpage
\begin{figure}[hbt!]
  \hspace*{-4cm}\includegraphics[width=100mm]{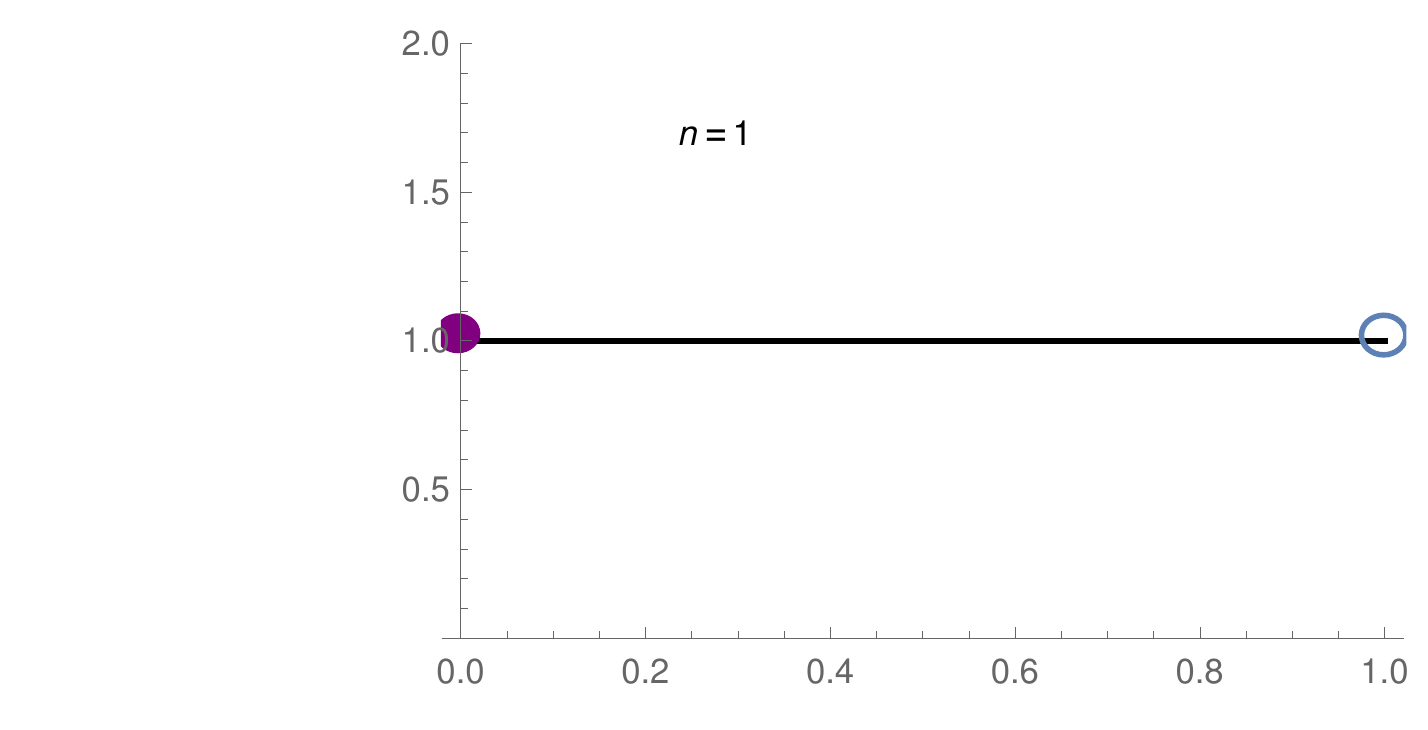}
 % \caption{}
  \label{fig:n1}
\end{figure}
\begin{figure}[hbt!]
  \includegraphics[width=80mm]{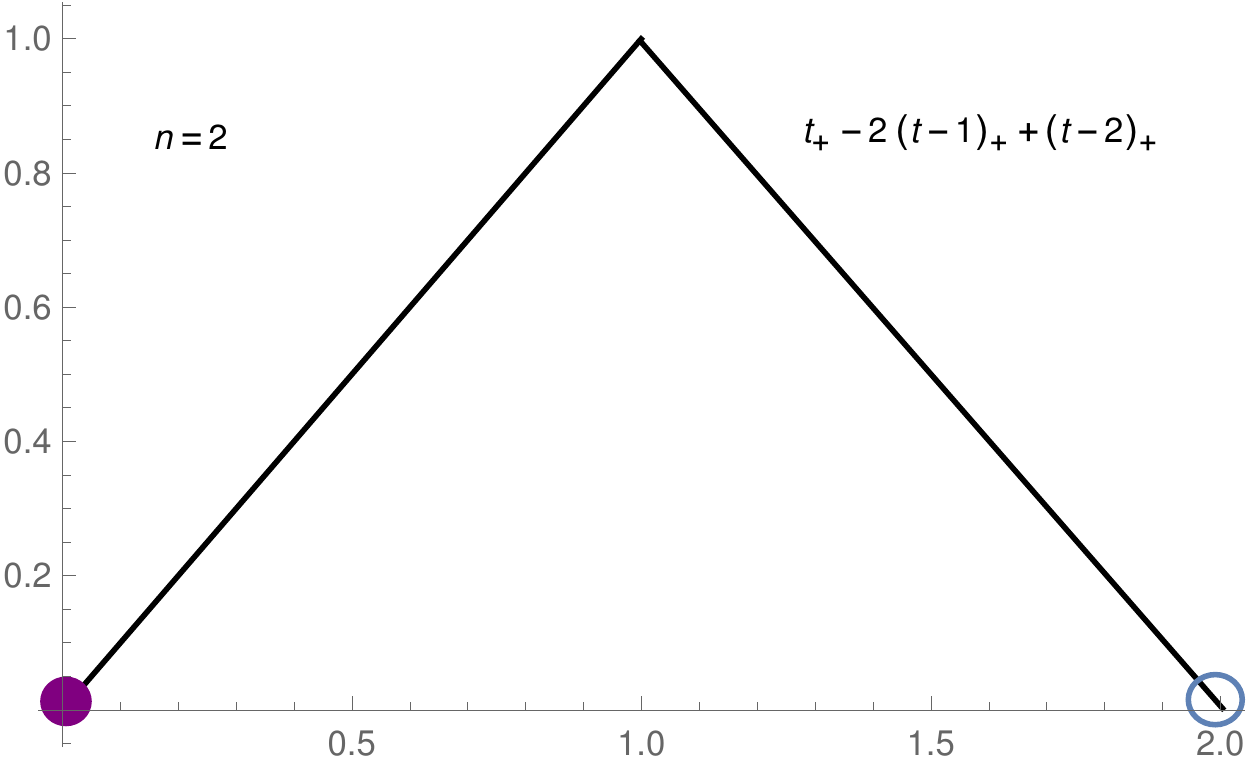}
 % \caption{}
  \label{fig:n2}
\end{figure}
\begin{figure}[hbt!]
  \includegraphics[width=80mm]{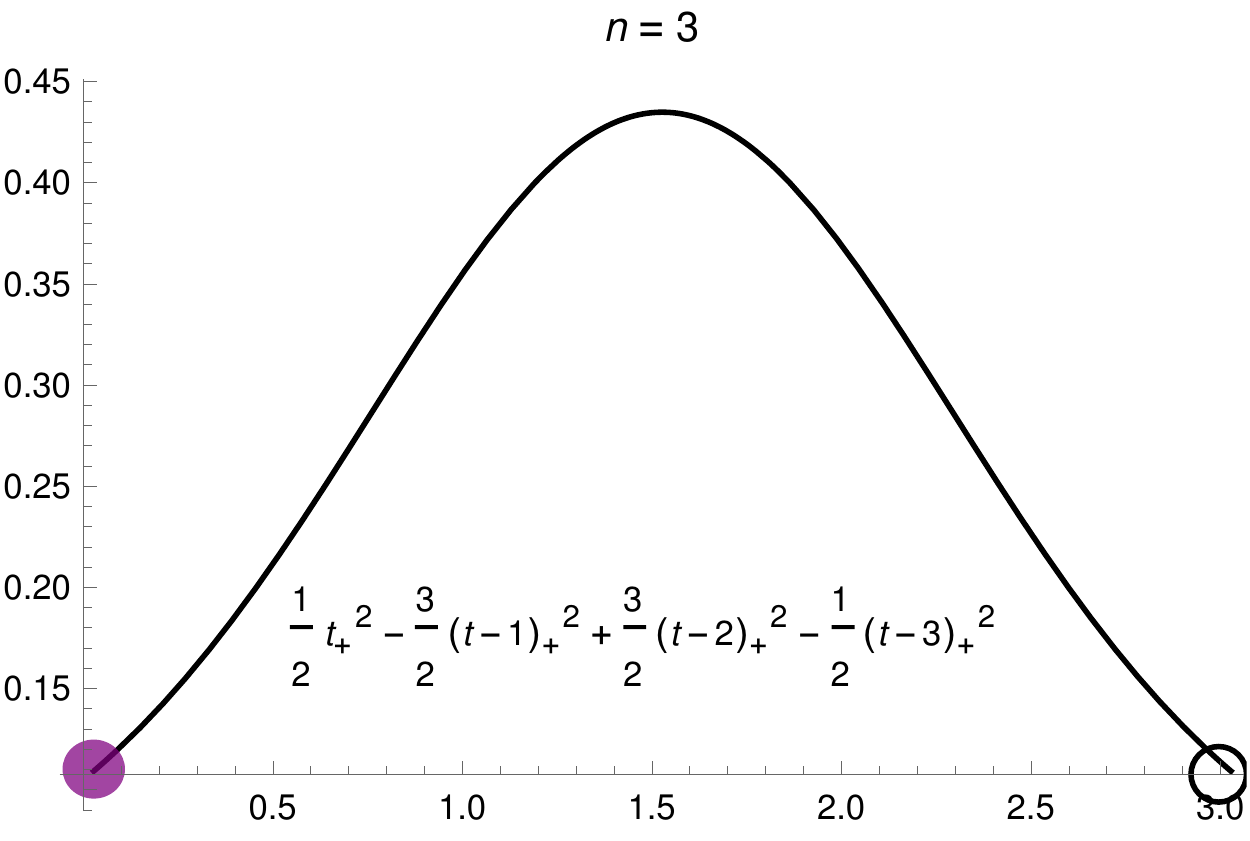}
 % \caption{}
  \label{fig:n3}
\end{figure}
In general it can be shown that $\rectFunc^{\otimes n}$ are unimodal functions and have maxima at $n/2$ and they are $(n-2)$ times differentiable on $(0,n)$ for $n\ge 3$.
We can write
\begin{align}
	\gamma_n(t)=\underbrace{\rectFunc*\cdots\rectFunc}_n(t)
\end{align}
and is given by the expression in (\ref{nconvrect}). These are universal functions. They may have other good applications in numerical analysis.

We can give an alternative expression for $\gamma_n$. Note that 

\begin{align*}
	\widehat{J^n f}(s)&=\widehat{M}_J(s)^n\hat{f}(s)\\
	&=\frac{1}{s^n}\hat{f}(s)
\end{align*}
and by Lemma \ref{lem2} with $a=0$ and $\gamma=n-1$.

\begin{align}
	J^n f(t)&=\int^t_0\frac{(t-\tau)^{n-1}}{(n-1)!}f(\tau)\,d\tau
\end{align}
If we select $f\equiv 1$, then
\begin{align*}
	J^n f(t)&=\int^t_0\frac{(t-\tau)^{n-1}}{(n-1)!}\,d\tau
\end{align*}
and so 
\begin{align}
	J^{n-1}1(t)=\frac{1}{(n-1)!}t^{n-1}_{+}
\end{align}
and
\begin{align}
	\widehat{J^{n-1}}(s)=\frac{1}{s^n}
\end{align}

\begin{lem} We have
\begin{align}\label{2a}
	\gamma_n(t)&=\frac{\Delta^n}{(n-1)!}t^{n-1}_{+}
\end{align}	
\end{lem}
\begin{proof}
	This follows from
\begin{align*}
	\widehat{\gamma}_n(s)\widehat{M}_{\Delta}(s)^n\widehat{J^{n-1}1}(s)
\end{align*}
As
\begin{align*}
	\Delta&=I-L\\
	\Delta^n &=\sum^{\infty}_{r=0}(-1)^r\binom{n}{r}L^r
\end{align*}
and so (\ref{2a}) is equivalent to (\ref{nconvrect}). The last step to $g^{\otimes n}$ is to compute $\A^n$.
\end{proof}
\begin{lem}\label{lem5}
	\begin{align}
		A^n=\sum^{\infty}_{r=0}\beta^n_r L^r
	\end{align}
where $\{\beta^n_r\}$ are calculated as follows.

\begin{align}
	\beta^1_r = a_r,\quad r=0,1,2,\cdots
\end{align}

\begin{align}
	\beta^{n+1}_r=\sum^{n+1}_{l=0}a_l\beta^n_{r-l}
\end{align}
\end{lem}

\begin{proof}
	We start with
\begin{align*}
	\A^{n+1}&=\A\sum^{\infty}_{r=0}\beta^n_r L^r\\
	&=\left(\sum^{\infty}_{j=0}a_j L^j\right)\left(\sum^{\infty}_{r=0}\beta^n_r L^r\right)
\end{align*}
and equate coefficient of $L^r$ on both sides.
\end{proof}
\begin{cor}[Corollary to \ref{lem5}]
	We note that $\beta^n_r$ is the coefficient of $x^r$ in 
	\begin{align*}
&\		(a_0+a_1x+a_2x^2+\cdots)^n\\
= &\ \sum^{\infty}_{r=0}\beta^n_r x^r
	\end{align*}
and if we put $x=1$,
\begin{align}\label{eq33}
	\sum^{\infty}_{r=0}\beta^n_r = k^n
\end{align}
\end{cor}

\begin{thm}\label{thm1}
\begin{align}\label{eq34}
	g^{\otimes n}(t)=A^n\gamma_n(t)=\sum^{\infty}_{r=0}\beta^n_r(L^r\gamma_n)(t)
\end{align}
\end{thm}

\begin{rmk}
	The intuition of Theorem \ref{thm1} is that, if we want $n$-fold convolution, the way to do it is to apply $A$ $n$-times to the $\gamma$. Then $\gamma_n(t)$ is the $n$-fold convolution of the rectanglular function and itself. Then the operator $\A$ applies to $\gamma_n$, $n$ times.
\end{rmk}

\begin{thm}\label{thm2}
\begin{align}
	h(t)=\sum^{\infty}_{n=0}\sum^{n-1}_{r=0}\beta^n_r (L^r\gamma_n)(t)
\end{align}
\end{thm}

Instead of taking unit interval, we approximate in step of $\delta$. If you solve for the case $\delta=1$. Then Theorem \ref{thm2} gives us formula for $h$.

\section{Part II : Actual case for Part I} We now let
\begin{align}
	g_{\delta}(t)&=
\begin{cases}
	\beta_j, &\text{  }j\delta\le t<(j+1)\delta\\
	0, &\text{ otherwise }
\end{cases}
\end{align}
Let us now define some useful operators. First, define $L_{\delta}$ ($\delta>0$)

\begin{align}
	(L_{\delta}f)(t)=f(t-\delta)
\end{align}

\begin{align}\label{eq36}
	(S_{\delta}f)(t)=\frac{1}{\delta}f\left(\frac{t}{\delta}\right)
\end{align}

\begin{align}
	S^{-1}_{\delta}=S_{1/\delta}
\end{align}

\begin{lem}\label{lem6}
	\begin{align}
		L_{\delta}= S_{\delta}LS^{-1}_{\delta}
	\end{align}
\end{lem}

\begin{proof}
	For the left hand side,
\begin{align*}
	L_{\delta}f(t)=f(t-\delta)
\end{align*}
For the right hand side,
\begin{align*}
	f(t)&\stackrel{S_{\frac{1}{\delta}}}{\to}\delta f(t\delta)\stackrel{L}{\to}\delta f(\delta(t-1))=\delta f(\delta t-\delta)\stackrel{S_{\delta}}{\to}=\frac{1}{\delta}(\delta f(\delta \frac{t}{\delta}-\delta))=f(t-\delta)
\end{align*}
\end{proof}

Let $g$ be as in Part I. Then

\begin{align}
&\	S_{\delta}g(t)\notag\\
= &\ \frac{1}{\delta}g\left(\frac{t}{\delta}\right)\notag\\
= &\
\begin{cases}
	\frac{1}{\delta}a_j, &\text{ if }j\le \frac{t}{\delta}<j+1\\
	0, &\text{ otherwise }
\end{cases}
\end{align}
So
\begin{align*}
	g_{\delta} = S_{\delta}g
\end{align*}
when $a_j=\beta_j\delta$ for $j=0,1,\cdots$. Suppose we want to find the solution of 

\begin{align}\label{hdel}
	h_{\delta}(t)=g_{\delta}(t)+\int^t_0 h_{\delta}(\tau)g_{\delta}(t-\tau)\,d\tau
\end{align}
We show:
\begin{lem}\label{lem7}
\begin{align}
	h_{\delta}(t)&=\frac{1}{\delta}h\left(\frac{1}{\delta}\right),\quad\text{for}\quad t\ge 0
\end{align}
\end{lem}
\begin{proof}
 Note that $h\delta$ is the solution of (\ref{hdel}) but $g_{\delta}(t)=\frac{1}{\delta}g\left(\frac{t}{\delta}\right)$, then
 
\begin{align*}
    h_{\delta}(t)&=\frac{1}{\delta}g\left(\frac{t}{\delta}\right)+\int^t_0 h_{\delta}(\tau)\frac{1}{\delta}g\left(\frac{t-\tau}{\delta}\right)\,d\tau\\
    &= \frac{1}{\delta}g\left(\frac{t}{\delta}\right)+\int^t_0 h_{\delta}(\tau)\frac{1}{\delta}g\left(\frac{t}{\delta}-\frac{\tau}{\delta}\right)\,d\tau\\
    &= \frac{1}{\delta}g\left(\frac{t}{\delta}\right)+\int^t_0 h_{\delta}(u\delta)\frac{1}{\delta}g\left(\frac{t}{\delta}-u\right)\delta\,du\quad u=\tau/\delta\\
    &=\frac{1}{\delta}g\left(\frac{t}{\delta}\right)  +\int^t_0 h_{\delta}(u\delta)g\left(\frac{t}{\delta}-u\right)\,du
\end{align*}
 Therefore,
 
\begin{align*}
	h_{\delta}(t\delta)&=\frac{1}{\delta}g(t)+\int^{t}_0 h_{\delta}(u\delta)g(t-u)\,du
\end{align*}
or
\begin{align*}
	\delta h_{\delta} (t\delta)=g(t)+\int^t_0\delta h_{\delta} (u\tau)g(t-u)du
\end{align*}
and by uniqueness of solution of the Volterra equation (\ref{eq3}),
\begin{align*}
	\delta h_{\delta}(t\delta)=h(t)
\end{align*} 
which implies 
\begin{align*}
	\delta h_{\delta}(t)=h(t/\delta)
\end{align*} 
or
\begin{align*}
	 h_{\delta}(t)=\frac{1}{\delta} h\left(\frac{t}{\delta}\right)
\end{align*} 
and so
\begin{align*}
	h_{\delta}(t)=\frac{1}{\delta}h\left(\frac{t}{\delta}\right)=S_{\delta}h(t)
\end{align*}
using the definition of $S_{\delta}$ in (\ref{eq36}).

\end{proof}

\section{Applications}

The solution to equation (1) is

\begin{align}
	y(t)=f(t)+(h*f)(t)
\end{align}

\begin{eg}
	Suppose
\begin{align}
	f(t)=\sum^{N}_{i=1}w_i\delta(t-t_i)+f_1(t)
\end{align}
where $\{w_1,\cdots,w_N\}$ are some weights and $t_1<t_2<\cdots<t_N$, $\delta$ is the diract delta function and $f_1\in L^1_{\text{loc}}(\R_+)$ (This means that $f_1(t)=-$ for $t<0$ and for any $T>0$, $\int^T_0|f_1(t)|\,dt<\infty$). If $f_1\in L^1(\R_+)$, that is $\int^{\infty}_0|f(t)|\,dt<\infty$, then $f_1\in L^1_{\text{loc}}(\R_+)$. If $f(t)=2+\sin(t)$ for $t\ge 0$ and $=0$ for $t<0$, then $w_i=0$, $i=1,\cdots,N$ and $f_1(t)=2+\sin{t}\in L^1_{\text{loc}}(\R_+)$, $f_1\notin L^1(\R_+)$. Thus
\begin{align*}
	y(t)&=f(t)+\left[h*\left(\sum^N_{i=1}w_i\delta(\cdot-t_i)+f_1\right)\right](t)\notag\\
	&= f(t)+\sum^{N}_{i=1}w_i h(t_i)+(h*f_1)(t)
\end{align*}
as $f_1\in L^1_{\text{loc}}(\R_{+})$, then $(h*f_1)(t)$ is well defined (as $h\in L^1(\R_{+})$) and $(h*f_1)\in L^1_{\text{loc}}(\R_{+})$ and 
\begin{align}
	\int^T_0|h*f_1(t)|\,dt&\le \|h\|_{L^1(\R_+)}\int^T_0|f_1(t)|\,dt
\end{align}
This means that we need a formula for $(h*f_1)(t)$ when $f_1\in L^1{\text{loc}}(R_{+}).$
\end{eg}

\begin{lem}
Let $f\in L^1_{\text{loc}}(\R^+)$, then
\begin{align}
    (h_{\delta}*f)(t)=S_{\delta}(h*(S_{1/\delta}f))(t),\quad t\ge 0
\end{align}
\end{lem}

\begin{proof}
Let 
\begin{align*}
    J(t)&=(h_{\delta}*f)(t)\\
    &=\int^t_0 h_{\delta}(\tau) f(t-\tau)\,d\tau\\
    &=\int^t_0\frac{1}{\delta} h\left(\frac{\tau}{\delta}\right)f(t-\tau)\,d\tau\quad\text{by Lemma 7}\\
    &=\int^{t/\delta}_0 h(u) f(t-u\delta)\,du\quad u=\frac{\tau}{\delta}\\
    &= \int^{t/\delta}_0 h(u) f(t-u\delta)\,du
\end{align*}
So
\begin{align*}
    J(\delta t) &=\int^t_0 h(u) f(t\delta-u\delta)\,du
\end{align*}
and so 

\begin{align*}
    \delta J(\delta t) = \int^t_0 h(u)\delta f(\delta(t-u))\,du
\end{align*}

\begin{align*}
    S_{\frac{1}{\delta}}J(t) = \int^t_0 h(u) S_{\frac{1}{\delta}}f(t-u)\,du
\end{align*}
and so 
\begin{align*}
    J(t) = S_{\delta}(h* S_{1/\delta} f) (t)
\end{align*}

\end{proof}
This means that $h_{\delta}*f$ can be calculated via $h$.

\textbf{Comment}\\

We might want to vary $\delta>0$, we could vary $h_{\delta}$ or we could keep $h$ constant (once calculated) and vary $S_{\delta}f$. If 
\begin{align*}
    f(t)=2+\sin{t},
\end{align*}
then
\begin{align*}
    S_{\delta}f(t) &= 
\begin{cases}
    2+\frac{1}{\delta}\sin{\frac{t}{\delta}}.& t\ge 0\\
    0 & t<0
\end{cases} 
\end{align*}

\begin{thm}
If $f\in L^1_{\text{loc}}(\R^+)$

\begin{align*}
 &\   (h*f)(t)\\
 = &\ \sum^{\infty}_{n=0}\sum^{\infty}_{r=0}\beta^n_r L^r (\gamma_n * f)
\end{align*}
\end{thm}

\begin{proof}
We note that
\begin{align*}
  &\  \reallywidehat{(L^r\gamma_n)*f}(s)\\
  = &\ \widehat{L^r\gamma_n}(s)\hat{f}(s)\\
  = &\ \widehat{M_L}(s)^r\widehat{(\gamma_n * f)}(s)\\
  = &\ \reallywidehat{L^r(\gamma_n*f)}(s)
\end{align*}
and the theorem is proved.
\end{proof}    

\begin{cor}
    If one requires $(h*f)(t)$ for $t\le n$ (for some integer $n\ge 1$) then one may use $L^r$ with $r\le n-1$.
\end{cor}

\begin{cor}
    We have 
\begin{align}\label{eq48}
    (h_{\delta}*f)(t) &= S^{-1}_{\delta}\sum^{\infty}_{n=0}\sum^{n-1}_{r=0}\beta^n_r L^r (\gamma_n * S_{\delta}f)(t)
\end{align}
and if we only need values for $t\le n\delta$, then may use $L^r$ with $r\le n-1$.
\end{cor}

\begin{lem}
If $f_1$, $f_2$ are in $L^1_{\text{loc}}(\R^+)$, then
\begin{align}
    L(f_1*f_2)=(Lf_1)*f_2 = f_1 * Lf_2
\end{align}
\end{lem}

\begin{proof}
We have
\begin{align*}
&\    L(f_1*f_2)(t)\\
= &\ (f_1 * f_2)(t-1)\\
= &\ \int^{t-1}_0 f_1(\tau)f_2(t-1-\tau)\,d\tau\\
= &\ \int^t_1 f_1(u-1) f_2(t-u)\,du\quad (u=1+\tau)\\
= &\ \int^t_0 f_1(u-1) f_2(t-u)\,du\,\,\text{as}\,\,\,f_1(u-1)=0,\,\,\text{for}\,\,\,0\le u<1\\
= &\ \int^t_0 (Lf_1)(u) f_2 (t-u)\,du\\
= &\ \{(Lf_1)*f_2\}(t)
\end{align*}
as $f_1*f_2 = f_2*f_1$, the theorem follows.
\end{proof}

\begin{rmk}
    We have 
    
\begin{align}
    L^r(\gamma_n *f)=(L^r\gamma_n)*f
\end{align}
We could have $L^r\gamma_n$ functions tabulated (universally calculated) for $r$ and $n$, we only need $r\le n-1$ in (\ref{eq33}) and (\ref{eq48}) as $\gamma_n(t)=0$ for $t\ge n$
\end{rmk}

\begin{lem}
Let $f_1$, $f_2\in L^1_{\text{loc}}(\R^{+})$.

\begin{align}
    L_{\delta}(f_1*f_2)&=(L_{\delta}f_1)*f_2\\
    &=\delta_1*(L\delta f_2)
\end{align}
\end{lem}

\begin{proof}
We have
\begin{align*}
    L_{\delta}(f_1*f_2)(t) &= (f_1*f_2)(t-\delta)\\
    &= \int^{t-\delta}_0 f_1(\tau)f_2(t-\delta-\tau)d\tau\\
    &=\int^t_{\delta}f_1(u-\delta)f_2(t-u)\,du\quad u=\delta+\tau\\
    &=\int^t_0 f_1(u-\delta)f_2(t-u)\,du\quad\text{as $f_1(u-\delta)=0$ if $0\le t<\delta$} \\
    &= (L_{\delta}f)*f_2 (t)
\end{align*}
and as before.
\end{proof}
We now try to simplify (\ref{eq48}) stated earlier by utilising the operators defined in Lemma \ref{lem7}
\begin{align}
    h_{\delta}&= S_{\delta}h\notag\\
    h_{\delta}*f &= (S_{\delta}h)*f
\end{align}

\begin{lem}
Let $f\in L^1(\R^+)$,

\begin{align}
    \widehat{S_{\delta}f}(s) = \hat{f}(s\delta)
\end{align}
\end{lem}

\begin{proof}
We have for $s>0$,

\begin{align*}
    \widehat{S_{\delta}f}(s)&=\int^{\infty}_0\frac{1}{\delta}f\left(\frac{t}{\delta}\right)e^{-st}\,dt\\
    &=\int^{\infty}_0\frac{1}{\delta}f(u)e^{-s\delta u}\delta\,du\quad t=\delta u\\
    &=\hat{f}(s\delta)
\end{align*}
\end{proof}
\begin{comment}
\begin{align*}
 &\   S^{-1}_{\delta}L^r (\gamma_n * S_{\delta}f)\\
 = &\ S^{-1}_{\delta} L^r S_{\delta}\circ S^{-1}_{\delta} (\gamma_n * S_{\delta}f)\\
 =&\ S^{-1}_{\delta} L^r S_{\delta}\circ S^{-1}_{\delta} (S_{\delta}f * \gamma_n)\\
 =&\ S^{-1}_{\delta} L^r S_{\delta} (f*S^{-1}_{\delta}\gamma_n)\\
 =&\ S^{-1}_{\delta} L^r S_{\delta} [(S^{-1}_{\delta}\gamma_n)*f]\\
 =&\ L^r_{\frac{1}{\delta}}[(S_{1/\delta}\gamma_n)*f] \quad\text{by lemma \ref{lem6}}
\end{align*}
This could be explored further. But a simplification of (\ref{eq48}) may not be possible.
\end{comment}

\section{Error Analysis}
We note that the results in (\ref{eq6}) and (\ref{eq7}). If $f\in L^1(\R^+)$, then\footnote{See proof in appendix.}

\begin{align}\label{eq53}
    \|h*f-h_{\delta}*f\|_1&\le \frac{\|f\|_1}{(1-k)^2}\|g-g_a\|_1
\end{align}

\begin{align}\label{eq54}
    \|h*f-h_{\delta}*f\|_{\infty}&\le \frac{\|f\|_1}{(1-k)^2}\|g-g_a\|_{\infty}
\end{align}
if $f\in L^1_{\text{loc}}(\R^+)$. Then (\ref{eq54}) becomes
\begin{align}\label{eq55}
 &\   \sup_{0\le t\le T}|h* f(t)-h_{\delta}*f(t)|\notag\\
 \le &\ \frac{\int^T_0|f(t)|dt}{1-k}\|g-g_a\|_{\infty}
\end{align}
which was proved in earlier notes.

\section{Examples: Power Law and Rayleigh Kernel}
\begin{eg}
    Let 
\begin{align}
    g(t) &=
\begin{cases}
    \frac{k\theta c^{\theta}}{(c+t)^{1+\theta}}&\text{if } t\ge 0\notag\\
    0 &\text{if }t<0
\end{cases}
\end{align}
 
\end{eg}

We set 
\begin{align*}
  \beta_j = g(j\delta)  \qquad j=0,1,\cdots
\end{align*}
\begin{figure}[h!] \includegraphics[width=80mm]{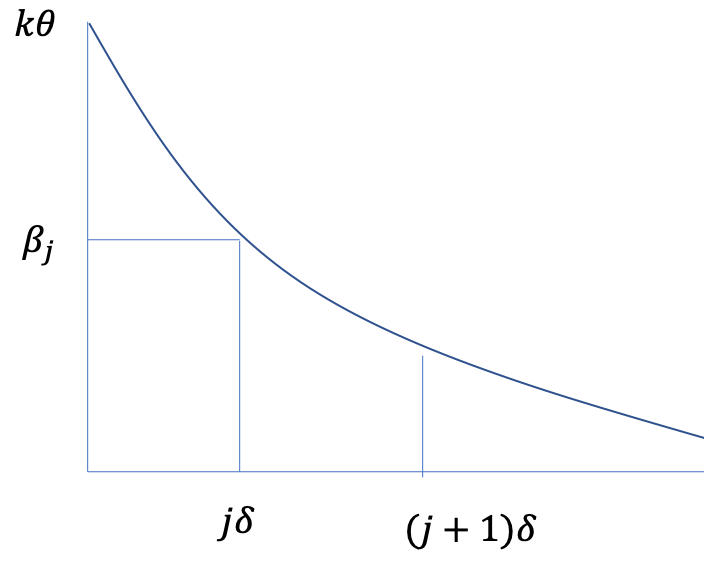}
%  \label{fig:pl}
\end{figure}    
then for $j\delta\le t<(j+1)\delta$

\begin{align*}
&\    |g(t)-g(j\delta)|\\
\le &\ |t-j\delta|\max_{j\delta\le t\le (j+1)\delta}|g'(t)|
\end{align*}

\begin{align*}
    g'(t)=\frac{-k\theta(1+\theta)}{(c+t)^{2+\theta}}
\end{align*}
and so $|g'(t)|\le k\theta (1+\theta)$ for all $t\ge 0$.

So 
\begin{align*}
 &\   |g(t)-g(j\delta)|\\
    \le &\ |t-j\delta|k\theta(1+\theta)\\
    \le &\ \delta k\theta (1+\theta) 
\end{align*}
So $g_{\delta}(t)$ with this choice of $\beta_j$ satisfies

\begin{align}\label{eq57}
    \|g-g_{\delta}\|_{\infty}&\le k\theta (1+\theta)\delta
\end{align}
So
\begin{align}\label{eq58}
    \|h-h_{\delta}\|_{\infty}&\le \frac{k\theta (1+\theta)\delta}{1-k}
\end{align}
and for $t\le T$.

\begin{align}\label{59}
    |h*f(t)-h_{\delta}*f(t)|&\le \frac{k\theta(1+\theta)\delta}{1-k}\int^T_0|f(t)|\,dt
\end{align}

\begin{eg}
    \begin{align}
        g(t)=\frac{kt}{\sigma^2}e^{-\frac{t^2}{2\sigma^2}}
    \end{align}
    
\begin{figure}[h!] \includegraphics[width=100mm]{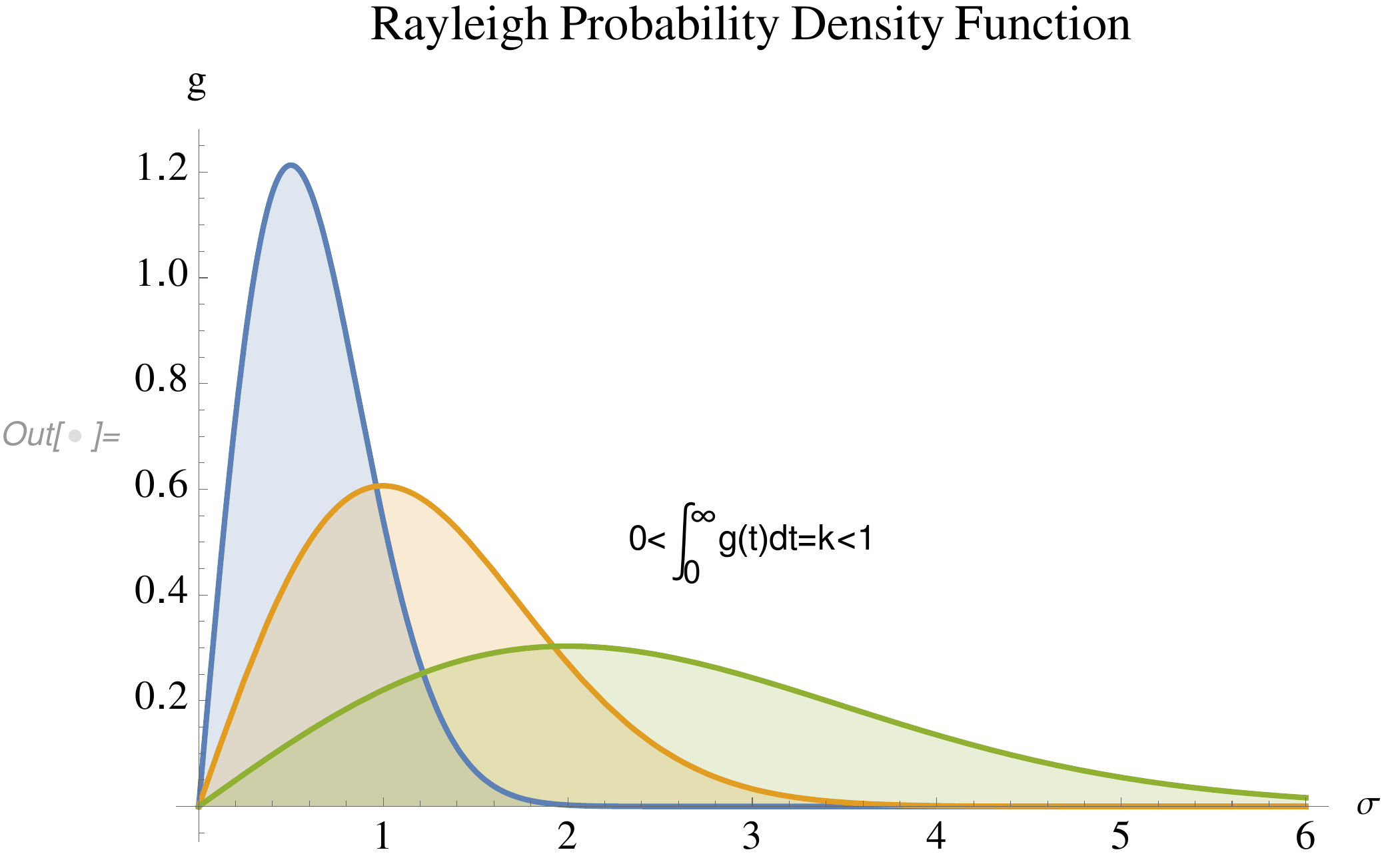}
%  \label{fig:pl}
\end{figure}    
    \begin{align*}
        g'(t)=\frac{k}{\sigma^2}[1-\frac{t^2}{\sigma^2}]e^{-\frac{t^2}{2\sigma^2}}
    \end{align*}
So
\begin{align*}
    |g'(t)|&\le \frac{k}{\sigma^2}\frac{1+t^2/\sigma^2}{e^{t^2/2\sigma^2}}\\
    &\le \frac{k}{\sigma^2}\frac{1+t^2/\sigma^2}{1+t^2/2\sigma^2}\\
    &\frac{2k}{\sigma^2}\frac{1+t^2/\sigma^2}{2+t^2/2\sigma^2}\le \frac{2k}{\sigma^2}
\end{align*}
Again choose 
\begin{align}\label{61}
    \beta_j = g(j\delta),\quad j=0,1,\cdots
\end{align}
and then 
\begin{align}\label{62}
    |g(t)-g_{\delta}(t)|\le \frac{2k}{\sigma^2}\cdot\delta
\end{align}
and the error analysis follows the same argument before.
\end{eg}

\begin{eg}
    Let
\begin{align}
    g(t)&=
\begin{cases}
    k &\text{if }0\le k<1\\
    0\quad &\text{otherwise}
\end{cases}
\end{align}
Thus this is the case where $a_0=k$, $a_j=0$ for $j\ge 1$, so $\beta^n_r=0$ if $r\ge 1$ and $\beta^n_0=k^n$

\begin{align}\label{eq66}
    h(t)=\sum^{\infty}_{n=0}k^n \gamma_n(t)
\end{align}
as $\gamma_n(t)=0$ for $t\ge n$.  We only need a finite number of terms in (\ref{eq66}) to calculate $h(t)$ for $t\le T$. In fact for $t\le T$

\begin{align*}
    h(t) = \sum^{[T]}_{n=0}k^n \gamma_n(t)
\end{align*}
\end{eg}

\section{Numerical implementation}

Compute $\beta^n_r$ coefficients for $r\le n-1$, all rest equal 0 for $n=1,2,\cdots$. Solve the function $\gamma_n(t)$ for each $n$ (MATLAB or C++) so that their values can be called. See equation (\ref{nconvrect}). If $k$ is small, we can drop out terms in (\ref{eq33}) with a small error that can be estimated. may lead to dropping further terms in calculation. No approximation is required, because for $0\le t\le T$ only a finite number of terms of the series are non zero.

\section{Future outlook}
We will implement our explicit algorithm into software and experiment with its behaviour. We will also experiment with real-world data by using our generative integral equation model to predict for arbitrary time point. (See the ODE version in \cite{chen2018}. This will contribute to not only the current scientific computing literature but also benefit the Machine Learning community. 
\section{Acknowledgments}
This material was motivated from a problem in computational social science. We thank Behavioral Data Science group, especially Dr Marian-Andrei Rizoiu in facilitating discussions and supporting us with research environment.

\section{Appendix}

We now prove (\ref{eq6}) and (\ref{eq7})
If we write (\ref{eq1}) as:
\begin{align*}
	h_i=g_i+g_i * h_i,\quad i=1,2
\end{align*}

\begin{align*}
h_1-h_2&=g_1-g_2+g_1*h_1-g_2*h_2\\
&=g_1-g_2+(g_1-g_2)*h_1+g_2*(h_1-h_2)
\end{align*}
So
\begin{align*}
	\|h_1-h_2\|&\le\|g_1-g_2\|+\|h_1\|\|g_1-g_2\|+\|g_2\|\|h_1-h_2\|
\end{align*}
So
\begin{align*}
	\|h_1-h_2\|&\le\frac{1}{1-\|g_2\|}\left(1+\|h_1\|\right)\|g_1-g_2\|\\
	&\le\frac{1}{1-\|g_2\|}\left(1+\frac{\|g_1\|}{1-\|g_1\|}\right)\|g_1-g_2\|\quad(4)\\
	&=\frac{\|g_1-g_2\|}{(1-\|g_1\|)(1-\|g_2\|)}\quad(5)
\end{align*}
where all norms are $L^1$ norms for which 
\begin{align*}
	\|\varphi*\psi\|_{L^1}&\le \|\varphi\|_{L^1} \|\psi\|_{L^1}
\end{align*}
If $g_1=g_2$ then $h_1=h_2$. This implies (\ref{eq1}) has an unique solution.

Also if $g_n\to g$ in $L^1$ then $h_n\to h$ in $L^1$. If we assume $\|g_n\|\le k<1$ for all $n\ge 1$, which leads to 
\begin{align*}\tag{6}
	\|h_n-h\|_{L^1}&\le \frac{1}{(1-k)^2}\|g_n-g\|_{L^2}
\end{align*}
(We do not want $\|g_n\|\to 1$ as $n\to\infty$)

We can also show convergence in $\text{sup}$ norms. Let us define
\begin{align*}
	\|\psi\|_t=\text{sup}_{0\le s\le t}|\psi(s)|
\end{align*}
Let $t>0$ and $s\in[0,t]$, then
\begin{align*}
\text{sup}_{0\le s\le t}h(s)&=\phi(s)+\int^s_0 h(\tau)g(s-\tau)\,d\tau\\
	&\le \|g\|_t+\|h\|_t\int^s_0 g(s-\tau)\,d\tau\\
	&\le \|g\|_t+\|h\|_t\int^t_0 g(\tau)\,d\tau\\
	&\le \|g\|_t+\|h\|_t\int^{\infty}_0 g(\tau)\,d\tau\\
	&\le \|g\|_t+k\|h\|_t
\end{align*}
as
\begin{align*}
\text{sup}_{0\le s\le t}h(s)&\le \|g\|_t
\end{align*}
So
\begin{align*}
	\|h\|_t&\le\|g\|_t+k\|h\|_t
\end{align*}
or
\begin{align*}
	\|h\|_t&\le\frac{1}{1-k}\|g\|_t
\end{align*}
and also
\begin{align*}\tag{7}
	\text{sup}_{t\ge 0} h(t)&\le\frac{1}{1-k}\text{sup}_{t\ge 0}g(t)
\end{align*}
Now we would like to derive (6) under $\text{sup}_t$, or $\|\cdot\|_t$ now.

For $0\le s\le t$, we have
\begin{align*}
	h_1(s)-h_2(s)=g_1(s)-g_2(s)+(g_1-g_2)*h_1(s)+g_2*(h_1-h_2)(s)
\end{align*}
Thus as before
\begin{align*}
	|h_1(s)-h_2(s)|&\le |g_1(s)-g_2(s)|+\|g_1-g_2\|_t\int^s_0 h_1(\tau)\,d\tau+\|h_1-h_2\|_t\int^s_0 g_2(\tau)\,d\tau\\
	&\le \|g_1-g_2\|_t(1+\|h_1\|_{L^1})
+\|g_1-g_2\|_t\|g_2\|_{L^1}
\end{align*}
and if $\|g_1\|$, $\|g_2\|\le k<1$, then
\begin{align*}
	1+\|h_1\|_{L^1}&\le 1+\frac{k}{1-k}=\frac{1}{1-k}
\end{align*}
and so 
\begin{align*}
	\|h_1-h_2\|_t&\le \|g_1-g_2\|_t\frac{1}{1-\|g_1\|_{L^1}}+\|h_1-h_2\|_t\|g_2\|_{L^1}
\end{align*}
and so
\begin{align*}\tag{8}
	\|h_1-h_2\|_t&\le\frac{\|g_1-g_2\|_t}{(1-\|g_1\|_{L^1})(1-\|g_2\|_{L^1})}
\end{align*}
and so
\begin{align*}\tag{9}
	\|h_1-h_2\|_t&\le\frac{1}{(1-k)^2}\|g_1-g_2\|_t
\end{align*}
and letting $t\to\infty$
\begin{align*}
	\|h_1-h_2\|_{\infty}&\le\frac{1}{(1-k)^2}\|g_1-g_2\|_{\infty}
\end{align*}
where $\|f\|_{\mathcal{D}}=\text{sup}_{t\ge 0}|f(t)|$
This implies that if $g_n\to g$ in sup norm then so does $h_n\to h$ in sup norm, provided $\|g_n\|\le k <1$ for all $n\ge 1$.

\subsection{Proofs of equations (\ref{eq53}) and (\ref{eq54})}

\begin{align*}
 &\   \|h*f-h_{\delta}*f\|\\
 \le &\ \|f\|_{L^1}\|h-h_{\delta}\|
\end{align*}
where $\|\cdot\|$ is either $L^1$ and $L^{\infty}$ norm. Then following the same arguments of proving (\ref{eq6}), one has
\begin{align*}
&\    \|h*f-h_{\delta}*f\|\\
\le &\ \frac{\|f\|_{L^1}}{(1-k)^2}\|g-g_{\delta}\|
\end{align*}
%Again, we have $\frac{1}{(1-k)^2}$ and $\frac{1}{1-k}$.
\bibliographystyle{plain}
\bibliography{hipper}
\end{document}